\documentclass[11pt,a4paper]{article}    
\usepackage{amsmath,amsfonts,amssymb,amsthm,amscd}
\usepackage[english]{babel}

\newtheorem{thm}{Theorem}
\newtheorem{prop}{Proposition}
\newtheorem{lemma}{Lemma}
\newtheorem{remark}{Remark}

\newcommand{\dsp}{\displaystyle}
\newcommand{\eps}{\varepsilon}

\newcommand{\R}{\mathbb{R}}

\newcommand{\supp}{\mathrm{supp}}
\newcommand{\sep}{\:|\:}

\numberwithin{equation}{section}

\begin{document}
 
\title{The Cauchy problem for the 3-D Vlasov-Poisson system with point charges}

 \author{Carlo Marchioro, Evelyne Miot and Mario Pulvirenti}

\date{\today}

\maketitle

\begin{abstract}
 In this paper we establish global existence and uniqueness of the solution to the three-dimensional Vlasov-Poisson 
system in presence
of point charges in case of repulsive interaction. The present analysis extends an analogeous two-dimensional 
result \cite{CM}.
\end{abstract}

\section{Introduction}

In this paper we study the time evolution of a three-dimensional system constituted by a continuous distribution
of electric charges, a plasma, coupled with $N$ charged point particles.

All the charges, as well as the plasma, have the same sign so that the interaction is repulsive. 
For simplicity we assume that the charges and the masses of the point particles are unitary. If $f=f(x,v,t)$ denotes 
the mass distribution of the plasma and $\{\xi_\alpha\}_{\alpha=1}^N$ are the positions of the point particles, the 
dynamics of the system
is described by the following system of equations
\begin{equation}
\label{eq:system}
 \begin{cases}
\dsp  \partial_t f+v\cdot \nabla_x f+(E+F)\cdot \nabla_v f=0\\ \vspace*{0.5em}
\dsp E(x,t)=\int_{\R^3} \frac{x-y}{|x-y|^3}\rho(y,t)\,dy\\\vspace*{0.5em}
\dsp \rho(x,t)=\int_{\R^3} f(x,v,t)\,dv\\\vspace*{0.5em}
\dsp F(x,t)=\sum_{\alpha=1}^N \frac{x-\xi_\alpha(t)}{|x-\xi_\alpha(t)|^3}\\\vspace*{0.5em}
\dsp \dot{\xi}_\alpha(t)=\eta_\alpha(t),\quad \alpha=1,\ldots,N\\
\dsp \dot{\eta}_\alpha(t)=E(\xi_\alpha(t),t)+\sum_{\substack{\beta=1\\\beta\neq \alpha}}^N 
\frac{\xi_\alpha(t)-\xi_\beta(t)}{|\xi_\alpha(t)-\xi_\beta(t)|^3}. 
 \end{cases}
\end{equation}
In absence of the charges, the system \eqref{eq:system} reduces to the well-known Vlasov-Poisson equation, which has been
widely investigated in the last years.

\medskip

The difficulty of the Cauchy problem associated to the Vlasov-Poisson problem increases with the dimension of the physical 
space. 

In two dimensions, satisfactory existence and uniqueness results go back to \cite{OkUk} and \cite{Horst}. 
The three-dimensional problem was solved in the nineties in \cite{Pf}, \cite{Sh}, \cite{Wo}, by a careful analysis
of characteristics associated to the Vlasov-Poisson system (Lagrangian point of view) or by estimating the moments of 
$f$ via a more genuine PDE technique \cite{LP} (eulerian point of view). We address the reader to the monograph \cite{G}
for a complete analysis of this equation and additional references.

\medskip

When point charges enter in the game, the situation changes drastically even if complete repulsivity is assumed. The extra singular
field could, in principle, produce extremely large velocities of the plasma particles and, in turn, a large spatial density
and a blow-up of the electric field in finite time. However we know that it is not the case in dimension two. Indeed, in \cite{CM}
it was shown a global existence and uniqueness result for the two-dimensional Cauchy problem associated to system 
\eqref{eq:system}. The basic ingredient of the analysis in \cite{CM} was the introduction of the energy of a trajectory of the
plasma in the reference frame of a suitable point charge. The advantage of using this energy function is two-fold. From one side it controls the motion. From the other, its time derivative along the
trajectory does not involve the singular part of the electric field. Combining this idea with the well-known fact that, 
in dimension two, the electric field generated by the plasma is linearly bounded by the maximal velocity of the particles, 
one can then 
conclude (see \cite{CM} for the details).

Our purpose here is to study the more complex three-dimensional problem. Our approach relies heavily on the adaptation of the method in 
\cite{Pf}, \cite{Sh} and \cite{Wo} to the present situation, for which the energy associated to a trajectory 
(see \eqref{def:en} or \eqref{def-m:en}) plays an essential role. We explain the main steps of this adaptation in the next section after some preliminary 
considerations have been presented.

\medskip

We mention that a related issue concerning the Cauchy problem for the Vlasov-Poisson equation, namely the instability arising from a singular perturbation of the field, has been recently analyzed in \cite{instab}, see also \cite{G1} or \cite{G2}. In that situation the extra singular field is due to reflecting boundary conditions.

\medskip

The plane of the paper is the following. In Section \ref{sec:one} we solve the Cauchy problem for one single charge ($N=1$).
This result is well suited to be easily extended to the full problem $N\geq 1$. This extension is done in 
Section \ref{sec:many}. Section \ref{sec:conclusion} is finally devoted to comments and criticism.

\section{Preliminaries and general strategy}

We use this section to fix the notations, to recall some well known estimates (see e.g. \cite{Wo}, \cite{G} and references 
quoted therein) and to illustrate the strategy we follow to treat the present problem.

\medskip

Let $f=f(x,v)$, $(x,v)\in \R^3\times \R^3$, be a probability density (namely $f\geq 0$ and $\int f\,dx\,dv=1$). We denote by
\begin{equation}
 \label{eq:kinetic}
 K_0=\frac{1}{2}\int |v|^2 f(x,v)\,dx\,dv
\end{equation}
the kinetic energy and by
$$ \rho(x)=\int f(x,v)\,dv$$
the spatial density.

Also, $K$ and $K_i,i=1,\ldots,$ will stand for positive constants depending only on the kinetic energy 
\eqref{eq:kinetic} and on $\|f\|_{L^\infty}$,
which we assume to be finite. Moreover $C$ will denote any numerical positive constant.

\medskip

For any $M\geq 0$, we have
\begin{equation*}
 \begin{split}
  \rho(x)&=\int_{|v|< M} f(x,v)\,dv+\int_{|v|\geq M}f(x,v)\,dv\\
&\leq \frac{4}{3} \pi M^3 \|f\|_{L^\infty}+\frac{1}{M^2}\int |v|^2 f(x,v)\,dv.
 \end{split}
\end{equation*}
Optimizing in $M$, we find
\begin{equation*}
 \rho(x)\leq C\|f\|_{L^\infty}\left( \int |v|^2 f(x,v)\,dv\right)^{3/5},
\end{equation*}
 whence
\begin{equation}
 \label{ineq:rho}
\|\rho\|_{L^{5/3}}\leq K_1.
\end{equation}
Moreover, defining for $R>0$
\begin{equation*}
 \rho_R(x)=\int_{|v|<R} f(x,v)\,dv,
\end{equation*}
we have by using H\"older inequality
\begin{equation*}
 \begin{split}
  \int \frac{\rho_R(x')}{|x-x'|^2}\,dx'&= \int_{|x-x'|< M}\frac{\rho_R(x')}{|x-x'|^2}\,dx'
+\int_{|x-x'|\geq M} \frac{\rho_R(x')}{|x-x'|^2}\,dx'\\
&\leq C \|\rho_R\|_{L^\infty}M+\|\rho_R\|_{L^{5/3}}\left( \int_{|x-x'|\geq M} \frac{dx'}{|x-x'|^5}\right)^{2/5}.
 \end{split}
\end{equation*}
Optimizing in $M$ and using \eqref{ineq:rho} we obtain
\begin{equation}
\label{ineq:rho2}
\int \frac{\rho_R(x')}{|x-x'|^2}\,dx'\leq K_2 \|\rho_R\|_{L^\infty}^{4/9}.
\end{equation}

Summarizing the previous considerations we are led to the

\begin{prop}
 \label{prop:prelim1}
There exists a constant $K>0$ (depending only on $\int |v|^2 f\,dx\,dv$ and $\|f\|_{L^\infty}$) for which
\begin{equation}
 \label{ineq:rho3}
\int \frac{\rho_R(x')}{|x-x'|^2}\,dx'\leq K R^{4/3}
\end{equation}
for all $R>0$.
\end{prop}

\begin{proof}
 Estimate \eqref{ineq:rho3} follows from \eqref{ineq:rho2}, realizing that $\|\rho_R\|_{\infty}\leq \frac{4}{3}
\pi R^3 \|f\|_{L^\infty}$.
\end{proof}

\bigskip

We now come to Problem \eqref{eq:system} and define the class of solutions we wish to deal with. 

 Let 
$\{(\xi_{\alpha0},\eta_{\alpha0})\}_{\alpha=1}^N$ denote the initial positions and velocities of the point charges. Let $f_0$ be a compactly supported probability density 
 satisfying, for some positive $\delta_0$,
 \begin{equation}
  \label{assumption:compact}
 \min \left\{ |x-\xi_{\alpha0}|\sep\: (x,v)\in \supp(f_0), \alpha=1,\ldots,N\right\}\geq \delta_0.
 \end{equation}
 
Let $T>0$.
We say that $(\xi_1,\ldots,\xi_N,f)$ is a solution of Problem \eqref{eq:system} on $[0,T]$ with initial datum
$(\xi_{\alpha0},\eta_{\alpha0},\alpha=1,\ldots,N,f_0)$ if   
\begin{equation}
\label{assumption:space}\xi_\alpha\in C^2([0,T]),\quad f\in L^\infty\left([0,T],L^1\cap L^\infty\right),\quad 
\rho \in L^\infty\left([0,T],L^\infty \right)\end{equation}
and for $t\in [0,T]$ we have
\begin{equation}
 \label{eq:edo2}
\begin{cases}
\dsp \dot{\xi}_\alpha(t)=\eta_\alpha(t)\\
\dsp \dot{\eta}_\alpha(t)=E(\xi_\alpha(t),t)+\sum_{\beta\neq \alpha}
\frac{\xi_\alpha(t)-\xi_\beta(t)}{|\xi_\alpha(t)-\xi_\beta(t)|^3}\\
\dsp(\xi_\alpha,\eta_\alpha)(0)=(\xi_{\alpha0},\eta_{\alpha0}).
\end{cases}
\end{equation}
Moreover 
\begin{equation}
\label{eq:transport}
f\left(X(x,v,0,t),V(x,v,0,t),t\right)=f_0(x,v),
\end{equation}
where for all $\tau,t\in [0,T]$ 
$$ (X,V)(\cdot,\cdot,\tau,t):\R^3\setminus\cup_\alpha\{\xi_\alpha(\tau)\}\times \R^3\to \R^3\setminus\cup_\alpha\{\xi_\alpha(t)\}\times \R^3$$
is an invertible flow such that
\begin{equation}
\label{assumption:lower}
 |X(x,v,\tau,t)-\xi_\alpha(t)|\geq \delta(T),\quad \forall(x,v)\in \supp(f(\tau)) 
\end{equation}
for some $\delta(T)>0$. It satisfies for $t\in [0,T]$
 \begin{equation}
\label{eq:edo}
\begin{cases}
  \dsp   \frac{d}{dt}X(x,v,\tau,t)=V(x,v,\tau,t) \\ 
\dsp  \frac{d}{dt}V(x,v,\tau,t)=E\left(X(x,v,\tau,t),t\right)+
\sum_{\alpha=1}^N\frac{X(x,v,\tau,t)-\xi_\alpha(t)}{|X(x,v,\tau,t)-\xi_\alpha(t)|^3}\\ 
\left(X,V\right)(x,v,\tau,\tau)=(x,v)\in \supp(f(\tau)).
 \end{cases}
\end{equation}

\medskip

Given $\rho\in L^\infty\left([0,T],L^1\cap L^\infty(\R^3)\right)$ it is a well-known fact (see e.g. \cite{LP}) that 
the corresponding field 
$E$ belongs to $L^\infty\left([0,T]\times \R^3\right)$. Moreover, it is almost-Lipschitz in the sense that for all $(x,y,t)\in \R^3\times \R^3\times [0,T]$
\begin{equation}
\label{eq:almost-lipschitz}
 |E(x,t)-E(y,t)|\leq C|x-y|\left(1+|\ln|x-y||\right).
\end{equation}
 In particular the solutions of the ODEs \eqref{eq:edo2} and \eqref{eq:edo} are uniquely defined
as long as the distance 
between the plasma
particles and the charges remains positive. 

Thanks to the hamiltonian structure of the system, the flow $(X,V)(0,t)$ preserves 
Lebesgue's measure on $\R^3\times \R^3$. As a result all the norms $\|f(t)\|_{L^p}$, $1\leq p\leq +\infty$ 
are preserved. Finally, it follows from \eqref{assumption:compact}, \eqref{assumption:lower} and
from the fact
that $E$ is bounded that the density $f(t)$ remains compactly supported for all $t\in[0,T]$.

\medskip

We define the total energy associated to Problem \eqref{eq:system} by
\begin{equation*}
 \begin{split}
H(f)&=\frac{1}{2}\int |v|^2f(x,v)\,dx\,dv+\frac{1}{2}\sum_{\alpha=1}^N |\eta_\alpha|^2\\
&+\sum_{\alpha=1}^N \int \frac{\rho(x)}{|x-\xi_\alpha|}\,dx
+\frac{1}{2}\iint \frac{\rho(x)\rho(y)}{|x-y|}\,dx\,dy
+\frac{1}{2}\sum_{\alpha\neq \beta}  \frac{1}{|\xi_\alpha-\xi_\beta|}.
\end{split}
\end{equation*}
One easily checks that for a solution to Problem \eqref{eq:system} $H(f(t))$ is finite and constant
on $[0,T]$.
Due to the positivity of the interaction, the kinetic energy \eqref{eq:kinetic} of the plasma part is bounded by 
$H(f(t))\equiv H(f(0))$. Therefore we may assume that the constant $K$ appearing in \eqref{ineq:rho3} does not depend on the time thanks to
the conservation of the energy and of $\|f\|_{L^\infty}$.

\medskip

Let us come back, for the moment, to the Vlasov-Poisson problem ignoring the point charges.

Denoting by
\begin{equation*}
P(t)=\text{sup}\left\{|v|\sep \:(x,v)\in \supp(f(t))\right\}
\end{equation*}
 we conclude by \eqref{ineq:rho3} that
\begin{equation}
 \label{ineq:P1}
P(t)\leq P(0)+ K\int_0^t P(s)^{4/3}\,ds.
\end{equation}
Indeed, the electric field computed on a characteristic $X=X(t)$ is bounded by
\begin{equation*}
\begin{split}
 |E(X(t),t)|&\leq \int \frac{\rho(x',t)}{|X(t)-x'|^2}\,dx'=\int \frac{\rho_{P(t)}(x',t)}{|X(t)-x'|^2}\,dx',
\end{split}
\end{equation*}
hence \eqref{ineq:P1} follows from \eqref{ineq:rho3}.

Obviously \eqref{ineq:P1} does not provide an a priori global bound for $P(t)$. A refined estimate allowing to 
solve the $3$-D problem has been obtained by considering the \emph{time averaging of the electric field} along a trajectory, see \cite{Pf}, \cite{Sh},
\cite{Wo} or \cite{G}. The basic idea consists in  partitioning the time interval $[0,T]$ into small pieces $(t_{i-1},t_i)$ 
for $i=1,2,\ldots,$ of small length $\Delta T=|t_{i-1}-t_i|$. For a given characteristic $(X,V)(t)$, one writes thanks to Liouville's theorem
\begin{equation}
 \label{eq:av1}
 \int_{t_{i-1}}^{t_i}dt\,|E\left(X(t),t\right)|\leq C \int_{t_{i-1}}^{t_i}dt\,\int f(y,w,t_{i-1})\frac{1}{|X(t)-Y(t)|^2}\,dy\,dw
\end{equation}
where $(Y,W)(t)$ is a characteristic leaving $(y,w)$ at time $t_{i-1}$.

Now, when $Y(t)$ is a trajectory of large velocity and large relative velocity at time $t_{i-1}$ 
(namely $|w|$ and $|V(t_{i-1})-w|$ are $\mathcal{O}(P^{4/3})$), we restrict our attention to the \emph{time integral}
\begin{equation}
 \label{eq:averaging}
\int_{t_{i-1}}^{t_i} \frac{dt}{|X(t)-Y(t)|^2}.
\end{equation}
Here $\Delta T$ is chosen so small that the relative velocity remains large in that time interval (\emph{stability property}).
Then the time integral \eqref{eq:averaging} can be computed almost explicitely, using that $X(t)-Y(t)$ essentially 
performs a free motion.
As a consequence the contribution of \eqref{eq:averaging} to the time integral of the electric field is shown to be smaller 
than $\mathcal{O}(P\Delta T)$ 
(see \cite{Pf}, \cite{Sh}, \cite{Wo} and \cite{G} or Lemma \ref{lemma:scat-plasma-plasma} below).
We call this contribution \emph{scattering plasma-plasma}.

The other contributions in \eqref{eq:av1} can then be handled by means of static estimates relying essentially on 
Proposition \ref{prop:prelim1} above.

\medskip

When a single point charge is present, the scenario changes dramatically, since the stability property for the trajectories of the plasma fails. Indeed the 
relative velocity of the plasma particles can change extremely fast if one of them collides with 
(or get very close to) the point charge. Nevertheless the interval of time in which the \emph{scattering charge-plasma} takes place is so small
that the contribution to the time integral \eqref{eq:averaging} can again be controlled (see Lemma \ref{lemma:scatt-plasma-charge} below).

One also has to remark that, instead of using the maximal velocity $P(t)$ as a control quantity, it is more convenient
to use as in \cite{CM} the energy of a trajectory, the 
time derivative of which cancels the singular part of the electric field.

\medskip

Once treated the single point charge-plasma in Section \ref{sec:one}, we 
can turn to the $N$-charges problem in Section \ref{sec:many}. Indeed, the choice of $\Delta T$ and of
all other parameters ensures that a plasma particle cannot get close to more than one point charge in every time interval
$(t_{i-1},t_i)$. As a consequence we can transfer the single charge analysis to the $N$-charges problem with minor 
modifications. 

\section{The plasma-charge model}
\label{sec:one}

Let $f=f(x,v,t)$ be the probability distribution of the plasma particles, and let $\xi$ and $\eta=\dot{\xi}$ denote
the position and the velocity of a single point particle of unitary charge. Problem \eqref{assumption:space}-\eqref{eq:edo} 
reads
\begin{equation}
\label{assumption:classe-1}
f\in L^\infty\left(L^1\cap L^\infty\right),
\quad \rho\in L^\infty\left(L^1\cap L^\infty\right),\quad E=\frac{x}{|x|^3}\ast \rho,
\end{equation}
with
\begin{equation}
 \label{eq:transport-one}
f\left(X(x,v,0,t),V(x,v,0,t),t\right)=f(x,v,0),\quad  (t,x,v)\in \R\times \R^3\times \R^3,
\end{equation}
where $(X,V)(x,v,0,t)=(X,V)(t)$ satisfies
\begin{equation}
\label{eq:syst1}
\begin{cases}
  \dsp   \dot{X}(t)=V(t) \\ 
\dsp  \dot{V}(t)=E\left(X(t),t\right)+\frac{X(t)-\xi(t)}{|X(t)-\xi(t)|^3}\\
\dsp (X,V)(0)=(x,v),\quad x\neq \xi(0)
 \end{cases}
\end{equation}
and 
\begin{equation}
\label{eq:syst2}
 \begin{cases}
  \dsp \dot{\xi}(t)=\eta(t)\\
\dsp \dot{\eta}(t)=E\left(\xi(t),t\right).
 \end{cases}
\end{equation}

The main result of this section is summarized in

\begin{thm}
\label{thm:main1}
 Let $f_0\in L^\infty$ be a compactly supported probability distribution. Let $(\xi_0,\eta_0)\in \R^3\times \R^3$. 
Assume that there exists some $\delta_0>0$ such that 
$$\mathrm{min}\left\{|x-\xi_0|\:\sep\:(x,v)\in \supp(f_0)\right\}\geq \delta_0.$$ 

For all time $T>0$ there exists a unique solution $(\xi,f)$ to Problem \eqref{assumption:classe-1}-
\eqref{eq:syst2} on $[0,T]$ with this 
initial datum.
\end{thm}

We shall prove that, assuming a solution to Problem \eqref{assumption:classe-1}-\eqref{eq:syst2} 
exists up to a fixed but arbitrary
time $T>0$, we have
\begin{equation}
 \label{ineq:H1}
\sup\left \{ |V(x,v,0,t)|+\frac{1}{|X(x,v,0,t)-\xi(t)|} \sep\quad   t\in [0,T],(x,v)\in \supp(f_0)\right\}\leq C(T)
\end{equation}
where $C(T)$ is a constant depending only on $f_0$ and $(\xi_0,\eta_0)$. As a consequence of \eqref{ineq:H1} the fact
that a unique solution to Problem \eqref{assumption:classe-1}-\eqref{eq:syst2}  does exist follows by 
rather standard arguments presented in Paragraph \ref{subsection:completed} at the end of the section.

\medskip

In what follows, $C_i$ and $K_i,i=0,1,\ldots,$ will denote positive constants depending only on 
$\|f_0\|_{L^\infty}$ and $H$, where $H$ denotes the global energy. 

\medskip

In the case of one single point charge, the energy reduces to
\begin{equation*}
\begin{split}
H\equiv\frac{1}{2}\int |v|^2f(x,v,t)&\,dx\,dv+\frac{1}{2}|\eta(t)|^2\\&+ \int \frac{\rho(x,t)}{|x-\xi(t)|}\,dx
+\frac{1}{2}\iint \frac{\rho(x,t)\rho(y,t)}{|x-y|}\,dx\,dy.
\end{split}
\end{equation*}
In particular, we have
\begin{equation}
 \label{ineq:eta}
|\eta(t)|\leq \sqrt{2H}.
\end{equation}
Following \cite{CM} we also introduce the pointwise energy of a plasma particle
\begin{equation}
\label{def:en}
h(x,v,t)=\frac{|v-\eta(t)|^2}{2}+\frac{1}{|x-\xi(t)|}+K_1,
\end{equation}
where $K_1$ is a large constant. A possible choice is
\begin{equation*}
 K_1\geq\max(8H,1).
\end{equation*}
In particular, 
in view of \eqref{ineq:eta} this choice ensures that for all $(x,v,t)$
\begin{equation}
 \label{ineq:H2}
|V(x,v,0,t)|\leq 2\sqrt{ h}\left(X(x,v,0,t),V(x,v,0,t),t\right).
\end{equation}
As already mentionned, the energy turns out to be a relevant quantity to control the motion, since it controls both the velocity and the distance from $\xi$ of the characteristic under consideration. We remark that $h$ is
uniformly bounded on $\supp(f_0)$ at time $0$.

In the following we shall use the short-hand notation $(X(t),V(t))=\left(X(x,v,0,t),V(x,v,0,t)\right)$ when
the initial condition $(x,v)$ is clear from the context. 

\medskip

Differentiating along the characteristics of the plasma particles and using \eqref{eq:syst1}-\eqref{eq:syst2} we find
\begin{equation*}
 \dot{h}\left(X(t),V(t),t\right)=\left( V(t)-\eta(t)\right)\cdot \left( E(X(t),t)-E(\xi(t),t)\right)
\end{equation*}
from which
\begin{equation}
 \label{ineq:h1}
  \Big|\frac{d}{dt} \sqrt{h}\left(X(t),V(t),t\right)\Big|\leq |E(\xi(t),t)|+|E(X(t),t)|.
\end{equation}
Note that the variation of $h$ is controlled by the smooth part of the electric field and this is, of course, crucial. We introduce the quantity 
\begin{equation*}
Q=Q_T=\sup\left\{\sqrt{h}\left(X(t),V(t),t\right) \sep\quad  t\in [0,T],(x,v)\in \supp (f_0)\right\}.
\end{equation*}
The remainder of this section is devoted to the proof of the following estimate 
from which \eqref{ineq:H1} follows immediately.
\begin{prop}
 \label{prop:mainQ}
We have
\begin{equation*}
Q_T\leq (Q_0+C_1)\exp(C_1 (1+T))\quad \forall T>0.
\end{equation*}
\end{prop}

\medskip

As explained earlier, the method relies on an suitable splitting of $[0,T]$ into small intervals. More precisely, we set
\begin{equation*}
 \Delta T=\frac{1}{K_2 Q_T},
\end{equation*}
where $K_2$ denotes a suitable large constant satisfying
\begin{equation*}
 K_2\geq 16\quad \text{and}\quad  \frac{8K}{K_2}<\frac{1}{8},
\end{equation*}
where $K$ (depending only on $\|f_0\|_{\infty}$ and $H$) is the constant appearing in Proposition \ref{prop:prelim1}. 

Next, if $\Delta T< T$ we set
\begin{equation*}
n=\left[\frac{T}{\Delta T}\right],\quad 
t_0=0, \quad t_n=T,\quad t_i=i\Delta T \quad \text{for}\quad i=0,\ldots,n-1,
\end{equation*}
so that
\begin{equation*}
 [0,T]=\bigcup_{i=1}^n [t_{i-1},t_i]\quad \text{with}\quad |t_i-t_{i-1}|\leq \Delta T.  
\end{equation*}

For $i=1,\ldots,n$ we define 
\begin{equation*}
 Q_i=\sup\left\{\sqrt{h}\left( X(t),V(t),t\right) \sep \quad t\in (t_{i-1},t_i),(x,v)\in \supp(f(t_{i-1}))\right\}
\end{equation*}
where here 
$\left(X(t),V(t)\right)=\left( X(x,v,t_{i-1},t),V(x,v,t_{i-1},t)\right)$ are the trajectories at time $t\geq t_{i-1}$,
leaving $(x,v)$ at time $t_{i-1}$. Finally, we set
\begin{equation*}
 Q_0=\sup\left \{\sqrt{h}(x,v,0) \sep \quad (x,v)\in \supp(f_0)\right\}.
\end{equation*}

In order to show Proposition \ref{prop:mainQ} we will first establish the following basic inequality
\begin{prop}
\label{prop:mainQi}Let $T>0$ such that $\Delta T< T$.
 We have
\begin{equation*}
Q_i\leq Q_{i-1}+C_2 Q_T \Delta T,\quad i=1,\ldots,n.
\end{equation*}
\end{prop}
We claim that Proposition \ref{prop:mainQ} follows from Proposition \ref{prop:mainQi}. Indeed, let us set 
$T_0=1/(4C_2)$. There are two cases.

If $\Delta T_0=1/(K_2Q_{T_0})<T_0$ then Proposition \ref{prop:mainQi} for $T_0$ implies that for all $i=1,\ldots,n$
\begin{equation*}
 \begin{split}
  Q_i\leq Q_0+C_2iQ_{T_0}\Delta T_0\leq Q_0+2C_2 T_0 Q_{T_0},
 \end{split}
\end{equation*}
hence
\begin{equation*}
 Q_{T_0}\leq \frac{Q_0}{1-2C_2 T_0}=2 Q_0.
\end{equation*}

Otherwise we have $\Delta T_0\geq T_0$, which means that
\begin{equation*}
 \begin{split}
  Q_{T_0}\leq \frac{1}{T_0K_2}= 4C_2 K_2^{-1}.
 \end{split}
\end{equation*}

In both cases we obtain
\begin{equation*}
\label{ineq:star}
 Q_{T_0}\leq 2Q_0+4C_2 K_2^{-1},
\end{equation*}
thus Proposition \ref{prop:mainQ} holds up to time $T_0$. Let now $T>T_0$ and $k=[\frac{T}{T_0}]$.
Since $T_0$ depends only on conserved quantities, we can iterate the previous arguments $k+1$ times to get
\begin{equation*}
\begin{split}
 Q_{T}\leq Q_{(k+1)T_0}\leq 2^{k+1}Q_0+4C_2 K_2^{-1}\sum_{j=0}^k 2^j\leq 2^{T/T_0+1}(Q_0+2C_2K_2^{-1})
\end{split}
\end{equation*}
and the conclusion follows.

\bigskip

We now come to the main ingredients for proving Proposition \ref{prop:mainQi}.
We observe preliminary that, without loss of generality, we may assume
\begin{equation}
\label{ineq:assumption}
 Q_i\geq K_3\geq 1
\end{equation}
where $K_3$ is a constant depending only on $K_1$ and $K$ (therefore only on $\|f\|_{\infty}$ and $H$) 
which will be specified in the course of
the proof of Lemma \ref{lemma:scatt-plasma-charge} below. Indeed, otherwise we have
$$Q_i\leq K_3\leq K_3K_2\frac{1}{K_2}\leq Q_{i-1}+C_2Q\Delta T$$
provided that $C_2\geq K_3 K_2$, and Proposition \ref{prop:mainQi} follows.

Next, we notice that, 
by virtue of \eqref{ineq:h1}, 
\eqref{ineq:H2} and \eqref{ineq:rho3} we have
\begin{equation}
 \label{ineq:H3}
\sqrt{h}\left(X(t),V(t),t\right)\leq \sqrt{h}(x,v,t_{i-1})+8K Q_i^{4/3}\Delta T
\end{equation}
for $t\in [t_{i-1},t_i]$ and $(x,v)\in \supp(f(t_{i-1}))$. Now, consider a trajectory $\left(X(t),V(t)\right)$ satisfying 
$$\sqrt{h}\left(X(\overline{t}),V(\overline{t}),\overline{t}\right)=Q_i$$ 
for some $\overline{t}\in[t_{i-1},t_{i}]$. We then have by definition of $\Delta T$
\begin{equation}
\label{ineq:che}
 \sqrt{h}(x,v,t_{i-1})\geq Q_i-\frac{8K}{K_2}Q_i^{1/3}\geq \frac{Q_i}{2}.
\end{equation}
Therefore to control $Q_i$ it suffices to control the energy of those trajectories for which \eqref{ineq:che} holds.

\medskip

In order to bound the time integral on the right-hand side of \eqref{ineq:h1} we have to evaluate the integrals
\begin{equation}
\label{eq:int}
 \int_{t_{i-1}}^{t_i} \frac{dt}{|X(t)-Y(t)|^2}\quad \text{and} \quad \int_{t_{i-1}}^{t_i} \frac{dt}{|\xi(t)-Y(t)|^2}
\end{equation}
for high energy trajectories $X(t)$, $Y(t)$ which could possibly make very small the denominators in \eqref{eq:int}. There are
various situations:
\begin{enumerate}
\item[(i)] Both $X$ and $Y$ are far from $\xi$ (Lemma \ref{lemma:scat-plasma-plasma}),
 \item[(ii)] $X$ is close to $\xi$ (Lemmas \ref{lemma:scatt-plasma-charge} and \ref{lemma:scatt-plasma-charge-bis}),
\item[(iii)] $X$ is far from $\xi$ but $Y$ is close to $\xi$, hence $X$ and $Y$ are far from each other (Lemma \ref{lemma:est2}).
\end{enumerate}
We shall handle each of these situations separately, achieving thereby the dynamical part of the proof. The remainder of the 
proof relies on rather straightforward estimates in phase-space.

\subsection{Preliminary estimates}

We start by establishing a lemma concerning the plasma-charge
scattering. 

\begin{lemma}
\label{lemma:fac}
 For any $(y,w)\in \supp(f(t_{i-1}))$ we have, with $(Y,W)(t)=(Y, W)(y,w,t_{i-1},t)$ 
\begin{equation*}
\int_{t_{i-1}}^{t_i} \frac{dt}{|Y(t)-\xi(t)|^2}\leq (2\sqrt{2}+1)Q_i.
\end{equation*}

\end{lemma}

\begin{proof}
 Setting $\ell(t)=|Y(t)-\xi(t)|$, we differentiate 
\begin{equation}
 \label{eq:l'}
\dot{\ell}=\frac{(Y-\xi)}{|Y-\xi|}\cdot(W-\eta),
\end{equation}
then 
\begin{equation*}
\begin{split}
\ddot{\ell}=\frac{|W-\eta|^2}{|Y-\xi|}+\frac{1}{\ell^2}+
\frac{(Y-\xi)}{|Y-\xi|}\cdot\left(E(Y)-E(\xi)\right)
-\frac{\left[(Y-\xi)\cdot(W-\eta)\right]^2}{|Y-\xi|^3}.
\end{split}
\end{equation*}
 Proposition \ref{prop:prelim1} and Cauchy-Schwarz inequality yield
\begin{equation*}
\ddot{\ell}\geq \frac{1}{\ell^2}-8K Q_i^{4/3}.
\end{equation*}
Therefore
\begin{equation*}
 \begin{split}
  \int_{t_{i-1}}^{t_i} dt\,\frac{1}{\ell^2(t)}&\leq \dot{\ell}(t_i)-\dot{\ell}(t_{i-1})+8K \Delta T Q_i^{4/3}\\
&\leq |W-\eta|(t_i)+|W-\eta|(t_{i-1})+\frac{8K}{K_2} Q_i^{1/3}.
 \end{split}
\end{equation*}
By definition of $K_2$ and since the first two terms in the right-hand side of the previous inequality are bounded by 
$\sqrt{2}Q_i$ we conclude the proof.
\end{proof}
We now introduce the quantities
\begin{equation}
 \label{def:1}
R_i=Q_i^{3/4}\quad \text{and}\quad \delta_i=Q_i^{-7/8}.
\end{equation}
Note that $R_i$ corresponds to the maximal radius in the velocity space for which \eqref{ineq:rho3} actually works, 
namely yields a linear estimate in $Q_i$.

As already mentionned, the choice of the parameters ensures stability for the quantity $\sqrt{h}$.
\begin{lemma}
 \label{lemma:stabh}Let $(y,w)\in \supp(f(t_{i-1}))$
and $(Y,W)(t)=\left(Y,W\right)(y,w,t_{i-1},t)$. 
 
If $\sqrt{h}(y,w,t_{i-1})\geq  R_i$ then
\begin{equation}
\label{ineq:stabh1}
 \sqrt{h}\left(Y(t),W(t),t\right)\geq \frac{1}{2}R_i,\quad \forall t\in [t_{i-1},t_i].
\end{equation}
 
 If $\sqrt{h}(y,w,t_{i-1})\leq  R_i$ then
\begin{equation}
\label{ineq:stabh2}
 \sqrt{h}\left(Y(t),W(t),t\right)\leq \frac{3}{2}R_i,\quad \forall t\in [t_{i-1},t_i].
\end{equation}

\end{lemma}

\begin{proof}
 This is a straightforward consequence of \eqref{ineq:h1}, observing that by choice of $\Delta T$ we have $8KQ_i^{4/3}\Delta T\leq R_i/2$.
\end{proof}

The quantity $\delta_i$ is the radius of a protection sphere around the charge $\xi$, outside which the singular field created
by $\xi$ is relatively moderate, namely $\mathcal{O}(Q_i^{7/4})$.

The next lemma deals with the plasma-plasma scattering when the influence of the charge is not too large. The control of the 
corresponding time integral is the same as in \cite{Pf}, \cite{Sh}, \cite{Wo} and \cite{G}. We shortly repeat it 
for completeness
and also because, for the present problem, we need a different choice of the parameters.

\begin{lemma}
\label{lemma:scat-plasma-plasma}Let $l_i>0$.
 Assume that there exists a time interval $$J=(\overline{t}_{i-1},\overline{t}_i)\subset (t_{i-1},t_i)$$ such that, 
for all $t\in J$
we have, with $\left(X,V\right)(t)=\left(X,V\right)(x,v,t_{i-1},t)$ and 
$\left(Y,W\right)(t)=\left(Y,W\right)(y,w,t_{i-1},t)$, where $(x,v)$ and $(y,w)\in \supp(f(t_{i-1}))$ 
\begin{equation}
 \label{ineq:scat1}
\inf_{t\in J}\min \left\{ |X(t)-\xi(t)|,|Y(t)-\xi(t)|\right\}> \delta_i
\end{equation}
and
\begin{equation}
\label{ineq:scat3} 
|V(\overline{t}_{i-1})-W(\overline{t}_{i-1})|>R_i.
\end{equation}
Then
\begin{equation}
 \label{ineq:lemma1}
\int_{\overline{t}_{i-1}}^{\overline{t}_i} dt\frac{\chi(|X(t)-Y(t)|>l_i)}{|X(t)-Y(t)|^2}\leq \frac{C_3}{l_i R_i}.
\end{equation}
\end{lemma}
\begin{remark}
 We shall use this estimate with the choice $l_i=Q_i^{-2}$.
\end{remark}

\begin{proof}
 Let $t_0\in [\overline{t}_{i-1},\overline{t}_i]$ be the minimizer of $|X(t)-Y(t)|$. Suppose for the moment that
$t_0\in (\overline{t}_{i-1},\overline{t}_i)$. Setting $D(t)=X(t)-Y(t)$, we have for $t\in (t_0,\overline{t}_i)$
\begin{equation}
 \label{eq:D}
D(t)=D(t_0)+\dot{D}(t_0)(t-t_0)+\rho(t),
\end{equation}
where
\begin{equation}
\begin{split}
 \label{eq:reste}
\rho(t)=\int_{t_0}^t ds\,(t-s)\Big( \frac{X(s)-\xi(s)}{|X(s)-\xi(s)|^3}&- \frac{Y(s)-\xi(s)}{|Y(s)-\xi(s)|^3}\\
&+E\left(X(s),s\right)-E\left(Y(s),s\right)\Big).
\end{split}
\end{equation}
By virtue of \eqref{ineq:scat1}, Proposition \ref{prop:prelim1} and the 
definition \eqref{def:1} of $R_i$ and $\delta_i$ we have
\begin{equation*}
\begin{split}
|\rho(t)|&\leq \frac{1}{2}(t-t_0)\left( 2 Q_i^{7/4}+8K Q_i^{4/3}\right)\Delta T\\
&\leq \frac{1}{2} (t-t_0)\left(\frac{2+8K}{K_2}\right) Q_i^{3/4}\\
&\leq \frac{1}{8}(t-t_0)Q_i^{3/4}, 
\end{split}
\end{equation*}
hence
\begin{equation}
 \label{ineq:reste}
|\rho(t)|\leq \frac{1}{8}(t-t_0)R_i.
\end{equation}

On the other hand, by \eqref{ineq:scat1} and \eqref{ineq:scat3} it holds
\begin{equation*}
\begin{split}
|\dot{D}(t_0)|&\geq |\dot{D}(\overline{t}_{i-1})|-\int_{\overline{t}_{i-1}}^{\overline{t}_{i}}|\ddot{D}(s)|\,ds\\
&\geq R_i-(2Q_i^{7/4}+8K Q_i^{4/3})\Delta T\\
&\geq R_i-\left(\frac{2+8K}{K_2}\right) Q_i^{3/4},
\end{split}
\end{equation*}
hence
\begin{equation}
 \label{ineq:D}
|\dot{D}(t_0)|\geq \frac{R_i}{2}.
\end{equation}

We remark that the parameter $\delta_i$ has precisely been chosen so as to ensure the above stability 
property, and that we have used the fact that $K_2$ is sufficiently large with respect to $K$.

We have
\begin{equation*}
|D(t)|\geq |D(t_0)+\dot{D}(t_0)(t-t_0)|-|\rho(t)|.
\end{equation*}
Since $D(t_0)$ and $\dot{D}(t_0)$ are orthogonal,  
\begin{equation}
\label{ineq:orth}
 |D(t_0)+\dot{D}(t_0)(t-t_0)|^2\geq  |D(t_0)|^2+|\dot{D}(t_0)|^2(t-t_0)^2\geq |\dot{D}(t_0)|^2(t-t_0)^2.
\end{equation}
Hence we have by \eqref{ineq:reste} and \eqref{ineq:D}
\begin{equation*}
\begin{split}
|D(t_0)+\dot{D}(t_0)(t-t_0)|&\geq |\dot{D}(t_0)|(t-t_0)\geq \frac{1}{2} R_i (t-t_0)\geq 4|\rho(t)|,
\end{split}
\end{equation*}
so that
\begin{equation*}
\label{ineq:D3}
 |D(t)|^2\geq \frac{9}{16}|\dot{D}(t_0)|^2(t-t_0)^2\geq \frac{9}{64} R_i^2(t-t_0)^2.
\end{equation*}
Finally,
\begin{equation*}
\begin{split}
 \int_{t_0}^{\overline{t}_i}dt\;\frac{\chi(|D(t)|>l_i)}{|D(t)|^2}&
\leq \int_{t_0}^{\infty}dt\;\frac{\chi(|D(t)|>l_i)}{|D(t)|^2}\\
&\leq \int_{t_0\leq t\leq t_0+ \frac{8l_i}{3R_i}}\frac{dt}{l_i^2}
+\int_{t\geq t_0+ \frac{8l_i}{3R_i}}\frac{64}{9R_i^2}\frac{dt}{(t-t_0)^2}\\
&\leq \frac{C_4}{l_i R_i}.
\end{split}
\end{equation*}

For the integral on the time interval $(\overline{t}_{i-1},t_0)$ we use the time reversal.

When $t_0=\overline{t}_{i-1}$ we use the same method, observing that $D(t_0)\cdot \dot{D}(t_0)\geq 0$ so that
\eqref{ineq:orth} still holds. Finally if $t_0=\overline{t}_i$ we use the time reversal.
\end{proof}

Next we need to control the charge-plasma scattering. Basically our aim is to show that the time spent by a trajectory
in the protection sphere $B(\xi(t),\delta_i)$ is very small. For proving this we apply the virial theorem argument,
introducing 
\begin{equation}
 \label{def:I}
I(t)=\frac{1}{2} |Y(t)-\xi(t)|^2.
\end{equation}
Differentiating we get
\begin{equation}
\label{ineq:I1}
 \dot{I}=(Y-\xi)\cdot (W-\eta)
\end{equation}
and
\begin{equation}
\label{ineq:I2}
 \ddot{I}=|W-\eta|^2+\frac{1}{|Y-\xi|}+(Y-\xi)\cdot \left(E(Y)-E(\xi)\right).
\end{equation}

\begin{lemma}
 \label{lemma:scatt-plasma-charge}
For $(y,w)\in \supp(f(t_{i-1}))$, suppose that
\begin{equation*}
 \sqrt{h}(y,w,t_{i-1})> R_i.
\end{equation*}
Consider $(Y,W)(t)=(Y,W)(y,w,t_{i-1},t)$. Then the set
\begin{equation*}
 J=\{t\in (t_{i-1},t_i)\sep\quad  |Y(t)-\xi(t)|<\delta_i\}
\end{equation*}
is connected. Moreover,
\begin{equation}
\label{est:m1}
 \mathrm{meas}(J)\leq C_5 Q_i^{-13/8}.
\end{equation}
\end{lemma}

\begin{proof} Let $t_0\in \overline{J}$ be a minimizer for $I(t)$. By \eqref{ineq:I1},
\begin{equation*}
 |\dot{I}(t)|\leq \sqrt{2I(t)}|W(t)-\eta(t)|,
\end{equation*}
 hence
\begin{equation*}
 |\dot{\sqrt{I}(t)}|\leq Q_i.
\end{equation*}
For $t\in [t_0,t_i)$ we obtain
\begin{equation*}
\begin{split} 
\sqrt{I(t)}&\leq \sqrt{I(t_0)}+Q_i\Delta T\leq \frac{1}{\sqrt{2}}Q_i^{-7/8}+\frac{1}{K_2},
\end{split}
\end{equation*}
therefore 
\begin{equation}
 \label{ineq:I3} |I(t)|\leq 1.
\end{equation}
Moreover, by Lemma \ref{lemma:stabh} we have for $t\in [t_0,t_i)$
\begin{equation*}
\label{ineq:stab1}
 \sqrt{h}\left(Y(t),W(t),t\right)\geq  \frac{R_i}{2}.
\end{equation*}
Then by \eqref{ineq:I2} and \eqref{ineq:I3} we have for $t\in [t_0,t_i)$
 \begin{equation*}
  \begin{split}
   \ddot{I}(t)&\geq h\left(Y(t),W(t),t\right)-K_1-|Y(t)-\xi(t)||E(Y(t),t)-E(\xi(t),t)|\\
 &\geq \frac{1}{4}R_i^2-K_1-8K\sqrt{2}Q_i^{4/3}\\
&= \frac{1}{4}Q_i^{3/2}-K_1-8K\sqrt{2} Q_i^{4/3}.
  \end{split}
 \end{equation*}
Now, we observe that for $K_3>0$ sufficiently large (depending only on $K_1$ and $K$) we have by 
assumption \eqref{ineq:assumption}
\begin{equation}
 \label{ineq:I4}
\ddot{I}(t)\geq \frac{1}{8}R_i^2.
\end{equation}

Consider now $(t_-,t_+)\subset J$ a maximal connected component containing $t_0$. If $t_0\in [t_-,t_+)$, $\dot{I}(t_0)\geq 0$
(if $t_0=t_+$ we use the same argument via the time reversal). Then
\begin{equation*}
 \dot{I}(t)\geq \dot{I}(t_0)+\frac{1}{8}R_i^2(t-t_0)\geq 0,\quad \forall t\in [t_0,t_i).
\end{equation*}
Since $I$ is increasing from $t_0$ up to $t_i$, the trajectories cannot reenter in the protection sphere once escaped. 
Therefore $J=(t_-,t_+)$ is connected.

Next, integrating twice \eqref{ineq:I4} in time and using that $\dot{I}(t_0)\geq 0$ we get
\begin{equation*}
 \frac{1}{2}\delta_i^2\geq I(t)\geq I(t_0)+\frac{1}{16}R_i^2(t-t_0)^2,\quad \forall t\in J,
\end{equation*}
so that
\begin{equation}
\label{ineq:m1}
 (t-t_0)^2\leq 8 Q_i^{-13/4},\quad \forall t\in J,
\end{equation}
and \eqref{est:m1} is proved.

\end{proof}

We finally obtain the following variant of Lemma \ref{lemma:scatt-plasma-charge}
\begin{lemma}
 \label{lemma:scatt-plasma-charge-bis}
For $(y,w)\in \supp(f(t_{i-1}))$, suppose that
\begin{equation*}
 \sqrt{h}(y,w,t_{i-1})> \frac{Q_i}{2}.
\end{equation*}
Consider $(Y,W)(t)=(Y,W)(y,w,t_{i-1},t)$. Then the set
\begin{equation*}
 J=\{t\in (t_{i-1},t_i)|\:|Y(t)-\xi(t)|<2\delta_i\}
\end{equation*}
is connected.
Moreover,
\begin{equation}
\label{est:m2-bis}
 \mathrm{meas}(J)\leq C_6 Q_i^{-15/8}.
\end{equation}
\end{lemma}
\begin{proof} 
 It suffices to follow the proof of Lemma \ref{lemma:scatt-plasma-charge} step by step, observing that 
estimate \eqref{ineq:m1} can even be improved if $h(y,w,t_{i-1})\geq Q_i^2/4$. Indeed, we have in this case 
$\ddot{I}(t)\geq Q_i^2/8$ for $t\in (t_{i-1},t_i)$ and everything goes exactly as before leading to \eqref{est:m2-bis}.
\end{proof}

This concludes the dynamical preparation. We now come to the proof of Proposition \ref{prop:mainQi}.

\subsection{Proof of Proposition \ref{prop:mainQi}}

In view of \eqref{ineq:h1}, in order to show Proposition \ref{prop:mainQi} we need
to control the integrals
\begin{equation*}
\mathcal{J}_1=\int_{t_{i-1}}^{t_i}dt\, |E\left(\xi(t),t\right)|\quad 
\text{and}\quad \mathcal{J}_2=\int_{t_{i-1}}^{t_i}dt\,|E\left(X(t),t\right)|.
\end{equation*}

\begin{lemma}
 \label{lemma:est1}We have
\begin{equation*}
 \mathcal{J}_1\leq C_7 Q \Delta T.
\end{equation*}
\end{lemma}

\begin{proof}
Setting $(Y,W)(t)=(Y,W)(y,w,t_{i-1},t)$ we have
\begin{equation}
 \label{ineq:J1}
 \begin{split}
\mathcal{J}_1&\leq \int_{t_{i-1}}^{t_i}\,dt  
\int_{(y,w)| \sqrt{h}(y,w,t_{i-1})\leq R_i}\,dy\,dw f(y,w,t_{i-1})\frac{1}{|Y(t)-\xi(t)|^2}\\
&+\int_{(y,w)| \sqrt{h}(y,w,t_{i-1})\geq R_i}\,dy\,dwf(y,w,t_{i-1})\int_{t_{i-1}}^{t_i}dt\,\frac{1}{|Y(t)-\xi(t)|^2}.
\end{split}
\end{equation}

Note that, by virtue of the stability property for $\sqrt{h}$ (see Lemma \ref{lemma:stabh}) and by \eqref{ineq:H2}
\begin{equation*}
 |W(t)|\leq 3R_i,\quad \forall t\in (t_{i-1},t_i)
\end{equation*}
when $\sqrt{h}(y,w,t_{i-1})\leq R_i$.
Therefore, thanks to Liouville's theorem the first term in the right-hand side of \eqref{ineq:J1} can be controlled by
\begin{equation*}
 \int_{t_{i-1}}^{t_i}\,dt
\int_{|\overline{w}|\leq 3 R_i}d\overline{y}
\,d\overline{w}\,f(\overline{y},\overline{w},t)\frac{1}{|\overline{y}-\xi(t)|^2}.
\end{equation*}
By virtue of Proposition \ref{prop:prelim1} with $R$ replaced by $3R_i$ this latter is bounded by
\begin{equation}
 \label{ineq:est2} 
C_8 R_i^{4/3}\Delta T=C_8Q_i \Delta T.
\end{equation}

We now turn to the second integral in \eqref{ineq:J1}. By virtue of Lemma \ref{lemma:fac} we can bound it by
\begin{equation*}
 \begin{split}
5Q_i\int_{(y,w)|\:h(y,w,t_{i-1})> R_i^2} &dy\,dw\,f(y,w,t_{i-1})\\
&\leq 5 Q_iQ_i^{-3/2}\int dy\,dw\,h(y,w,t_{i-1})f(y,w,t_{i-1})\\
&\leq C_{9} Q_i^{-1/2}(H+1)\\
&\leq C_{10} Q\Delta T,
\end{split}
\end{equation*}
where we have used that $Q_i\geq K_3$ in the last inequality.
The conclusion follows.
\end{proof}

\begin{lemma}
 \label{lemma:est2}Let $(x,v)\in \supp(f(t_{i-1}))$ and 
$(X,V)(t)=(X,V)(x,v,t_{i-1},t)$ be a trajectory such that $h(x,v,t_{i-1})>Q_i^2/4$. Then
\begin{equation*}
 \mathcal{J}_2\leq C_{11} Q \Delta T.
\end{equation*}
\end{lemma}

\begin{proof}
 Let $(t_-,t_+)$ be the connected interval (if any) in which we have $|X(t)-\xi(t)|<2\delta_i$. By virtue of Proposition \ref{prop:prelim1} 
and Lemma \ref{lemma:scatt-plasma-charge-bis} (estimate \eqref{est:m2-bis})
\begin{equation*}
 \int_{t_-}^{t_+}dt\,|E\left(X(t),t\right)|\leq 4 K Q_i^{4/3} C_6 Q_i^{-15/8}\leq C_{12} Q \Delta T. 
\end{equation*}

It remains to control the integrals
\begin{equation*}
 \int_{t_{i-1}}^{t_-}dt\,|E\left(X(t),t\right)|\quad \text{and}\quad  
\int_{t_+}^{t_i}dt\,|E\left(X(t),t\right)|,
\end{equation*}
which can be handled (using again the time reversal) in the same way. We write
\begin{equation*}
 \begin{split}
  \int_{t_{i-1}}^{t_-}&dt \int dy\,dw\, f(y,w,t_{i-1})\frac{1}{|X(t)-Y(t)|^2}\\
&=\sum_{j=1}^4 \int_{S_j} dt\,dy\,dw\, f(y,w,t_{i-1})\frac{1}{|X(t)-Y(t)|^2}\\
&=\sum_{j=1}^4 \tilde{\mathcal{J}}_j,
 \end{split}
\end{equation*}
where
\begin{equation*}
 \begin{split}
  S_1&=\{(y,w,t)|\quad \sqrt{h}(y,w,t_{i-1})\leq R_i\},\\
S_2&=\{(y,w,t)|\quad \sqrt{h}(y,w,t_{i-1})>R_i\quad \text{and}\quad |X(t)-Y(t)|\leq l_i\},\\
S_3&=\{(y,w,t)|\quad \sqrt{h}(y,w,t_{i-1})>R_i, \quad |X(t)-Y(t)|> l_i\quad \text{and}\quad 
|Y(t)-\xi(t)|\leq \delta_i\}\\
S_4&=\{(y,w,t)|\quad \sqrt{h}(y,w,t_{i-1})>R_i, \quad |X(t)-Y(t)|> l_i\quad \text{and}\quad 
|Y(t)-\xi(t)|> \delta_i\},
 \end{split}
\end{equation*}
and where $$l_i=Q_i^{-2}.$$
Using stability for $\sqrt{h}$ (see Lemma \ref{lemma:stabh}) 
as well as static estimates, the first integral $\tilde{\mathcal{J}}_1$ can be estimated as before
in Lemma \ref{lemma:est1} (see \eqref{ineq:est2}), so that
\begin{equation*}
 \tilde{\mathcal{J}}_1\leq C_{13}Q_i\Delta T.
\end{equation*}

For the integral on $S_2$, following \cite{Pf}, \cite{Sh}, \cite{Wo} or \cite{G} we get
\begin{equation*}
 \tilde{\mathcal{J}}_2\leq \int_{t_{i-1}}^{t_-}dt \int_{|X(t)-\overline{y}|<l_i}
d\overline{y}\,\frac{\rho(\overline{y},t)}{|\overline{y}-X(t)|^2}\leq C_{14}Q_i^3l_i\Delta T
\leq C_{14}Q_i\Delta T .
\end{equation*}

Next, for the integral $\tilde{\mathcal{J}}_3$ we have
\begin{equation*}
 \tilde{\mathcal{J}}_3\leq \int_{h(y,w,t_{i-1})>R_i^2} dy\,dw\, f(y,w,t_{i-1})
\int_{t_{i-1}}^{t_-}dt\,
\frac{\chi(|Y(t)-\xi(t)|\leq\delta_i)}{|Y(t)-X(t)|^2}.
\end{equation*}
Since $|X(t)-\xi(t)|>2\delta_i$ in $(t_{i-1},t_-)$ we have in $S_3$
\begin{equation*}
 |Y(t)-X(t)|\geq |X(t)-\xi(t)|-|Y(t)-\xi(t)|\geq \delta_i.
\end{equation*}
Hence, in view of the conservation of the energy and of 
Lemma \ref{lemma:scatt-plasma-charge} (estimate \eqref{est:m1}) we have
\begin{equation*}
\begin{split}
 \tilde{\mathcal{J}}_3&
\leq \frac{1}{\delta_i^2}\int_{h(y,w,t_{i-1})>R_i^2} dy\,dw\, f(y,w,t_{i-1})\int_{t_{i-1}}^{t_-}dt\,\chi(|Y(t)-\xi(t)|\leq \delta_i)\\
&\leq C_{15}Q_i^{7/4} Q_i^{-3/2}Q_i^{-13/8}\\
&\leq C_{16}Q\Delta T.
\end{split}
\end{equation*}

We finally estimate the last integral $\tilde{\mathcal{J}}_4$. By virtue of Lemma \ref{lemma:scatt-plasma-charge}, 
for each $(y,w)$ such that $\sqrt{h}(y,w,t_{i-1})> R_i$ we may split the set 
$$\{t\in (t_{i-1},t_-)|\:|Y(t)-\xi(t)|\geq \delta_i\}$$ into
 two at most intervals $J^1(y,w)$ and $J^2(y,w)$ for which
\begin{equation*}
 \inf_{t\in J^k(y,w)}|Y(t)-\xi(t)|\geq \delta_i,\quad k=1,2.
\end{equation*}
Note that at least one of the extremal points of $J^1$ or $J^2$ has to coincide with at least one of the $t_{i-1}$ or $t_i$.
Hence we have  
\begin{equation*}
\tilde{\mathcal{J}}_4\leq \sum_{k=1}^2 \int_{h(y,w,t_{i-1})>R_i^2} dy\,dw\, f(y,w,t_{i-1})
\int_{J^k(y,w)}dt\,
\frac{\chi(|X(t)-Y(t)|>l_i)}{|Y(t)-X(t)|^2}.
\end{equation*}
It suffices to control the integral on $J^1(y,w)$ because the integral on $J^2(y,w)$ can be handled in the same way.

We set $J^1(y,w)=(\overline{t}_-,\overline{t}_+)$. Then we
further split the integration domain as follows
\begin{equation*}
 \begin{split}
\{(y,w)\sep \quad \sqrt{h}(y,w,t_{i-1})> R_i\}=S_4^{(1)}\cup S_4^{(2)},
\end{split}
\end{equation*}
where
\begin{equation*}
\begin{split}
S_4^{(1)}&= \{(y,w)|\quad \sqrt{h}(y,w,t_{i-1})> R_i\quad \text{and}\quad  |W(\overline{t}_{-})-V(\overline{t}_{-})|\leq R_i\},\\
S_4^{(2)}&= \{(y,w)|\quad \sqrt{h}(y,w,t_{i-1})> R_i\quad\text{and}\quad  |W(\overline{t}_{-})-V(\overline{t}_{-})|> R_i\}.
\end{split}
\end{equation*}
We recall that $W(t)$ denotes the velocity of the plasma particle leaving $(y,w)$ at time $t_{i-1}$.

First, for $(y,w)\in S_4^{(1)}$ we have 
\begin{equation*}
 |W(t)-V(t)|\leq\frac{3}{2} R_i,\quad \forall t\in J_1(y,w).
\end{equation*}
Indeed, both $X(t)$ and $Y(t)$ remain at least at distance $\delta_i$ of the charge
in $J_1(y,w)$, therefore
\begin{equation*}
\left| \frac{d}{dt}|W(t)-V(t)|\right|\leq 2 Q_i^{7/4}+8 K Q_i^{4/3},
\end{equation*}
so our choice of $\Delta T$ ensures stability for the velocities (see the proof of Lemma \ref{lemma:scat-plasma-plasma}).
Hence, applying Liouville's theorem we obtain
\begin{equation*}
\begin{split}
 \int_{S_4^{(1)}} dy\,dw\, f(&y,w,t_{i-1})
\int_{J^1(y,w)}dt\,
\frac{\chi(|X(t)-Y(t)|>l_i)}{|Y(t)-X(t)|^2}
\\
&\leq \int dy\,dw\, f(y,w,t_{i-1})
 \int_{t_{i-1}}^{t_i}dt\;
\frac{\chi(|V(t)-W(t)|<3R_i/2)}{|Y(t)-X(t)|^2}\\
&\leq 
 \int_{t_{i-1}}^{t_i}dt\;\int_{|\overline{w}-V(t)|\leq 3R_i/2} d\overline{y}\,d\overline{w}\, f(\overline{y},\overline{w},t)
\frac{1}{|\overline{y}-X(t)|^2}\\
&\leq C_{17}Q_i\Delta T,
\end{split}
\end{equation*}
where we have used Proposition \ref{prop:prelim1} in the last inequality.

Finally, let $(y,w)\in S_4^{(2)}$. Since $|X(t)-\xi(t)|\geq 2\delta _i >\delta_i$ on $J^1(y,w)$, we may now apply Lemma \ref{lemma:scat-plasma-plasma} to obtain  
\begin{equation*}
\begin{split}
 \int_{ J^1(y,w)}dt\,\frac{\chi(|X(t)-Y(t)|>l_i)}{|X(t)-Y(t)|^2}
&\leq \frac{C_{18}}{R_il_i}.
\end{split}
\end{equation*}
Since $l_i=Q_i^{-2}$ this yields
\begin{equation*}
\begin{split}
\int_{S_4^{(2)}} dy\,dw\, f(y,w,t_{i-1})
&\int_{J^1(y,w)}dt\,
\frac{\chi(|X(t)-Y(t)|>l_i)}{|Y(t)-X(t)|^2} \\&\leq \frac{C_{19}}{R_i^2}\cdot \frac{1}{R_i l_i}\leq 
C_{19}Q_i^{-9/4}Q_i^2\leq C_{20} Q\Delta T.
\end{split}
\end{equation*}
This concludes the proof of the Lemma.
\end{proof}

\bigskip

Combining Lemmas \ref{lemma:est1} and \ref{lemma:est2} we 
may finally turn to the

\noindent \textbf{Proof of Proposition \ref{prop:mainQi} completed}.

\medskip

Let $(X,V)(t)=\left(X,V\right)(x,v,t_{i-1},t)$ be a trajectory such that
$$\sqrt{h}\left(X(\overline{t}),V(\overline{t}),\overline{t}\right)=Q_i\quad \text{for some}\quad \overline{t}\in[t_{i-1},t_i].$$ By virtue of \eqref{ineq:che} we have
\begin{equation*}
 \sqrt{h}(x,v,t_{i-1})> \frac{ Q_i}{2}.
\end{equation*}

On the other hand, integrating first \eqref{ineq:h1} on $[t_{i-1},\overline{t}]$ and applying then Lemmas \ref{lemma:est1}
and \ref{lemma:est2} to the high-energy trajectory $(X,V)$ we obtain
\begin{equation*}
\begin{split}
 \sqrt{h}\left(X(\overline{t}),V(\overline{t}),\overline{t}\right)
&\leq \sqrt{h}(x,v,t_{i-1})+\int_{t_{i-1}}^{t_i}dt\,\left(|E\left(X(t),t\right)|+|E\left(\xi(t),t\right)|\right)\\
&\leq Q_{i-1}+C_{20}Q \Delta T,
\end{split}
\end{equation*}
whence
\begin{equation*}
 Q_i\leq Q_{i-1}+C_{20}Q\Delta T
\end{equation*}
and the conclusion follows.

\subsection{Proof of Theorem \ref{thm:main1} completed}
\label{subsection:completed}

We finally complete the proof of Theorem \ref{thm:main1} in the following way. We first regularize the 
Coulomb kernel $1/|x|$
at a small distance from the origin, say $\eps$, obtaining a solution 
$(\xi^\eps(t),\eta^\eps(t),f^\eps(t), X^\eps(t), V^\eps(t))$ of the corresponding regularized version of Problem 
\eqref{assumption:classe-1}-\eqref{eq:syst2}. The fact that the corresponding global energy $H^\eps$ is 
uniformly bounded by $H$ provides uniform bounds for the kinetic energy \eqref{eq:kinetic} of $f^\eps$ and for $|\eta^\eps|$ (see \eqref{ineq:eta}).

We choose $\eps$ smaller than $1/C(T)$, where $C(T)$ is the a priori bound appearing in 
\eqref{ineq:H1}. Next, we take $T_\eps$ maximal such that
\begin{equation*}
 \min\left\{|X^\eps(x,v,0,t)-\eta^\eps(t)|\:\sep\:(x,v)\in \supp(f_0)\right\}>\eps,\quad t\in[0,T_\eps).
\end{equation*}
On $[0,T_\eps)$, the extra field created by the charge 
coincides with the one of Problem \eqref{eq:system} for all trajectory starting from $\supp(f_0)$. Therefore the previous
analysis applies, yielding \eqref{ineq:H1}. We conclude that $T_\eps=T$, so that \eqref{ineq:H1} provides uniform $L^\infty$ bounds 
for $\rho_\eps$ and $E_\eps$ on $[0,T]$. It follows that $(\xi^\eps,\eta^\eps)$ 
is uniformly bounded and equicontinuous on $[0,T]$ and that $(X^\eps,V^\eps)$ 
is uniformly bounded and equicontinuous
on $\supp(f_0)\times[0,T]$. Hence one easily passes to the
limit $\eps\to 0$ to get existence of a global solution $(\xi,f)$ satisfying the desired conditions.

\medskip

Finally, uniqueness is achieved by means of a Gronwall inequality, using almost-Lipschitz regularity for $E$ 
(see \eqref{eq:almost-lipschitz}) and the lower bound
\eqref{assumption:lower} together with the assumption on the support of $f_0$. We omit the details.

\section{The case of many point charges}
\label{sec:many} 

The purpose of the present section is to extend the existence and uniqueness result of the previous section
to the case of many point charges. We will establish the following
\begin{thm}
 \label{thm:main-m}
Let $N\geq 1$. Let $f_0\in L^\infty$ be a compactly supported distribution and let 
$\{\xi_{\alpha0},\eta_{\alpha0}\}_{\alpha=1}^N$ such that $\xi_{\alpha0}\neq \xi_{\beta0}$
for $\alpha\neq \beta$.

 Assume that 
there exists some $\delta_0>0$ such that
$$\min\left\{
|x-\xi_{\alpha0}|\:\sep (x,v)\in \supp(f_0)\right\}\geq \delta_0.$$

 Then for all $T>0$, there exists a unique solution to Problem \eqref{assumption:space}-\eqref{eq:edo} on $[0,T]$ with these initial data.
\end{thm}

We proceed as in the proof of Theorem \ref{thm:main1}, considering a solution on $[0,T]$ obtained by regularizing
the singular interaction kernel in Problem \eqref{eq:system} and establishing estimates depending only on 
$\|f_0\|_{\infty}$ and $H$. Here again,  $C$, $C_i$ or $K_i,i=1,\ldots,$ 
denote positive constants depending only these quantities, $C$ changing possibly from a line to another. 

\medskip

First, we infer from the conservation of the total energy $H(f)$ that 
\begin{equation}
\label{est:many1}
 |\xi_\alpha(t)-\xi_\beta(t)|\geq  \lambda ,\quad \forall t\in [0,T],
\end{equation}
where $\lambda=1/(2H)$. Also,
\begin{equation}
 \label{est:many2}
 |\eta_\alpha(t)|\leq \sqrt{2H},\quad \forall t\in [0,T].
\end{equation}

For $\alpha=1,\ldots,N$ we introduce the energy related to the $\alpha$-th charge
\begin{equation}
\label{def-m:en}
 h_\alpha(x,v,t)=\frac{1}{2}|v-\eta_\alpha(t)|^2+\frac{1}{|x-\xi_\alpha(t)|}+K_1,
\end{equation}
where $K_1$ has already been defined in the previous section.

We set
\begin{equation*}
 Q=\max_{\alpha=1,\ldots,N}
\sup\left\{ \sqrt{h_\alpha}\left(X(x,v,0,t),V(x,v,0,t),t\right)\sep\quad t\in[0,T],(x,v)\in \supp(f_0)\right\}.
\end{equation*}

As before, we split the interval $[0,T]$ into small intervals $[t_{i-1},t_{i}]$ of length smaller than $\Delta T$, where
\begin{equation*}
 \Delta T=\frac{1}{K_2 Q}.
\end{equation*}
Here $K_2$, depending only on $H$ and $\|f_0\|_{\infty}$, is chosen in a similar way as in the previous section, and we assume moreover that 
$$K_2\geq \sqrt{2}\frac{48}{\lambda}.$$
Therefore, if $(X,V)(t)=(X,V)(x,v,t_{i-1},t)$ denotes a plasma particle leaving $(x,v)$ at time $t_{i-1}$ we have
thanks to \eqref{eq:l'}
\begin{equation}
\label{ineq-m:stab}
 \left|\frac{d}{dt}|X(t)-\xi_\alpha(t)|\right|\leq\sqrt{2}Q\leq \frac{\lambda}{48}\frac{1}{\Delta T},
\end{equation}
 so that for all $\alpha=1,\ldots,N$
\begin{equation}
\label{eq:isolated}
x\in B\left(\xi_\alpha(t_{i-1}),\frac{\lambda}{8}\right)\Rightarrow X(t)\in B\left(\xi_\alpha(t),\frac{7\lambda}{48}\right)
\quad\forall 
t\in [t_{i-1}, t_i]
\end{equation}
and
\begin{equation}
 \label{eq:isolated2}
x\in B\left(\xi_\alpha(t_{i-1}),\frac{\lambda}{8}\right)^c\Rightarrow X(t)\in 
B\left(\xi_\alpha(t),\frac{5\lambda}{48}\right)^c\quad\forall 
t\in [t_{i-1}, t_i].
\end{equation}
In view of \eqref{est:many1}, equation \eqref{eq:isolated} means 
that a plasma particle starting close to the 
$\alpha$-th charge at time $t_{i-1}$ cannot approach the other charges on $[t_{i-1},t_i]$. This property will enable us to isolate each charge
and to apply the previous analysis to the present case.

\medskip

For $i=1,\ldots,n$ we define
\begin{equation*}
 Q_i=\max_{\alpha=1,\ldots,N}
\sup\left\{\sqrt{h_\alpha}\left( X(t),V(t),t\right)\sep\: t\in (t_{i-1},t_i),(x,v)\in \supp(f(t_{i-1}))\right\}
\end{equation*}
where $(X,V)(t)=\left( X,V\right)(x,v,t_{i-1},t)$. Finally we set
\begin{equation*}
 Q_0=\max_{\alpha=1,\ldots,N}\sup\left\{\sqrt{h}(x,v,0)\sep \: (x,v)\in \supp(f_0)\right\}
\end{equation*}
so that
\begin{equation*}
 Q=\max_{i=0,\ldots,n}Q_i.
\end{equation*}

As explained in the previous section, Theorem \ref{thm:main-m} is a consequence of the following variant of 
Proposition \ref{prop:mainQi} for the present case
\begin{prop}
\label{prop:mainQi-m}
 Let $T>0$ such that $\Delta T<T$. For all $i=1,\ldots,n$ we have
\begin{equation*}
 Q_i\leq Q_{i-1}+C_1Q \Delta T.
\end{equation*}
\end{prop}

We will only present the proof of Proposition \ref{prop:mainQi-m}. Again we may assume that
\begin{equation*}
 Q_i\geq K_3\geq1.
\end{equation*}

We introduce the security balls
\begin{equation}
 B_\beta=B\left(\xi_\beta(t_{i-1}), \frac{\lambda}{8}\right),\quad \beta=1,\ldots,N
\end{equation}
and the complementary set
\begin{equation*}
 \Omega=\R^3\setminus \bigcup_{\beta=1}^N B_\beta.
\end{equation*}

Next, we set for $t\in[t_{i-1},t_i]$
\begin{equation}
 \label{def:part}
\begin{split}
f_\beta(x,v,t)=\chi_{B_\beta}f(x,v,t)
\quad \text{and}\quad \tilde{f}(x,v,t)&=\chi_{\Omega}f(x,v,t),
\end{split}
\end{equation}
so that $$f=\sum_{\beta=1}^Nf_\beta + \tilde{f},$$ and we denote by $E_\beta$ and $\tilde{E}$ the 
corresponding electric fields, so that $$E=\sum_{\beta=1}^N E_\beta+\tilde{E}.$$

\begin{lemma}
 \label{prop:alpha}
Let $\alpha\in \{1,\ldots,N\}.$ Let $(x,v)\in \supp(f(t_{i-1}))$ such that
\begin{equation*}
x\in B_\alpha \quad \mathrm{and}\quad h_\alpha(x,v,t_{i-1})\geq \frac{1}{4}Q_i^2.
\end{equation*}
Then 
\begin{equation*}
 \sqrt{h_\alpha}\left(X(t),V(t),t\right)\leq \sqrt{h_\alpha}(x,v,t_{i-1})+C_2Q\Delta T,\quad t\in [t_{i-1},t_i].
\end{equation*}
\end{lemma}

\begin{proof} Introducing the field
 \begin{equation*}
  F_\alpha(z,t)=\sum_{\beta\neq \alpha}\frac{z-\xi_\beta(t)}{|z-\xi_\beta(t)|^3},\quad z\neq \xi_\beta(t),
 \end{equation*}
we compute
\begin{equation}
\label{est-m:der}
\begin{split}
 \left|\frac{d}{dt} \sqrt{h_\alpha}\left(X(t),V(t),t\right)\right|&\leq |E\left(X(t),t\right)|+|E\left(\xi_\alpha(t),t\right)|
\\&+|F_\alpha\left(X(t),t\right)|+|F_\alpha\left(\xi_\alpha(t),t\right)|.
\end{split}
\end{equation}

Thanks to \eqref{eq:isolated} and \eqref{est:many1}, we observe that $X(t)$ remains
far from the charges $\xi_\beta(t)$ for $\beta\neq \alpha$. Hence the fields $F_\alpha(\xi_\alpha)$ and $F_\alpha(X)$ are 
bounded on $[t_{i-1},t_i]$ and we have
\begin{equation}
\label{est-m:0}
\int_{t_{i-1}}^{t_i}dt\,\left(|F_\alpha\left(\xi_\alpha(t),t\right)|+| F_\alpha\left(X(t),t\right)|\right)
\leq C\Delta T\leq CQ\Delta T.
\end{equation}

Moreover, since $\tilde{E}+\sum_{\beta\neq \alpha}E_\beta+F_\alpha$ is bounded by $\mathcal{O}(Q_i^{4/3})$ away from the 
$\alpha$-th charge, we may follow step by step the dynamical preparation performed in the previous section, 
simply replacing the quantities $\ell(t)$, $\sqrt{h}(t)$ and $I(t)$ by $\ell_\alpha(t)=|X(t)-\xi_\alpha(t)|$, 
$\sqrt{h_\alpha}(t)$ and $I_\alpha(t)=|X(t)-\xi_\alpha(t)|^2/2$.
In particular, adapting the proofs of Lemmas \ref{lemma:est1} and \ref{lemma:est2} we find 
\begin{equation}
\label{est-m:1}
 \int_{t_{i-1}}^{t_i} \,dt\left( |E_\alpha\left(\xi_\alpha(t),t\right)|+|E_\alpha\left(X(t),t\right)|\right)\leq C Q
 \Delta T.
\end{equation}

It remains to estimate the contributions of the other fields. Let $y\in B_\beta$. Thanks to \eqref{est:many1} and 
\eqref{eq:isolated} we have
\begin{equation*}
 |X(t)-Y(t)|\geq |\xi_\alpha(t)-\xi_\beta(t)|-|X(t)-\xi_\alpha(t)|-|Y(t)-\xi_\beta(t)|
\geq  \frac{17\lambda}{24},
\end{equation*}
hence $E_\beta(X)$ is bounded on $[t_{i-1},t_i]$. By \eqref{est:many1} and \eqref{eq:isolated2} $E_\beta(\xi_\alpha)$ is also bounded. Therefore
\begin{equation}
\label{est-m:3}
\sum_{\beta\neq \alpha} 
\int_{t_{i-1}}^{t_i}dt\, \left(|E_\beta\left(\xi_\alpha(t),t\right)|+|E_\beta\left(X(t),t\right)|\right)\leq CQ\Delta T. 
\end{equation}

We finally estimate the contribution of $\tilde{E}$. By \eqref{est:many1} and \eqref{eq:isolated2} $\tilde{E}(\xi_\beta)$ is bounded, thus
\begin{equation}
\label{est-m:5}
 \int_{t_{i-1}}^{t_i}dt\, |\tilde{E}\left(\xi_\alpha(t),t\right)|\leq C\Delta T\leq CQ \Delta T. 
\end{equation}

In order to estimate $\tilde{E}\left(X(t),t\right)$ we distinguish between two cases.

We assume first that $x\in B(\xi_\alpha(t_{i-1}),\lambda/16)$. For $y\in \Omega$ we then have $|X(t)-Y(t)|\geq \lambda/48$ by 
\eqref{est:many1} and \eqref{ineq-m:stab}. 
Hence $\tilde{E}(X)$ is bounded and we obtain
\begin{equation}
\label{est-m:6}
  \int_{t_{i-1}}^{t_i}dt\, |\tilde{E}\left(X(t),t\right)|\leq CQ \Delta T\quad \text{if}\quad 
x\in B\left(\xi_\alpha(t_{i-1}),\frac{\lambda}{16}\right). 
\end{equation}

Otherwise, we have $x\in B_\alpha\setminus B(\xi_\alpha(t_{i-1}),\lambda/16)$. Let $y\in \Omega$. 
In view of \eqref{est:many1} the particles $X(t)$ and $Y(t)$ both remain in 
$\R^3\setminus \cup_{\beta} B(\xi_\beta(t),\lambda/24)$ on $[t_{i-1},t_i]$. One may therefore neglect the plasma-charge interaction and only 
take into account the plasma-plasma interaction. Following the arguments in the case without charges (see \cite{Wo}), 
or adapting Lemmas \ref{lemma:scat-plasma-plasma} and \ref{lemma:est2} we obtain 
\begin{equation}
\label{est-m:7}
 \int_{t_{i-1}}^{t_i} dt\, |\tilde{E}\left(X(t),t\right)|\leq CQ\Delta T\quad \text{if}\quad 
 x\in B_\alpha\setminus B\left(\xi_\alpha(t_{i-1}),\frac{\lambda}{16}\right).
\end{equation}

Gathering estimates \eqref{est-m:0} to \eqref{est-m:7} and using \eqref{est-m:der} we are led to the conclusion of 
Proposition \ref{prop:alpha}.
\end{proof}
Our next result concerns the variation of the energies $h_\alpha$ away from the charges.
\begin{lemma}
 \label{prop:away} Let $(x,v)\in \supp(f(t_{i-1}))$ such that $x\in \Omega$.

For all $\alpha\in \{1,\ldots,N\}$ we have
\begin{equation*}
 \sqrt{h_\alpha}\left(X(t),V(t),t\right)\leq \sqrt{h_\alpha}(x,v,t_{i-1})+C_{3}Q\Delta T,\quad t\in (t_{i-1},t_i).
\end{equation*}
\end{lemma}

\begin{proof} We have $X(t)\in\R^3\setminus \cup_{\alpha} B(\xi_\alpha(t), 5\lambda/48)$.
Imitating the proof of Lemma \ref{prop:alpha}, it only remains to estimate 
\begin{equation*}
\int_{t_{i-1}}^{t_i}dt\, |E_\alpha\left(X(t),t\right)|
\end{equation*}
for all $\alpha$.
We proceed as before, writing 
$$B_\alpha=B\left(\xi_\alpha(t_{i-1}),\frac{\lambda}{16}\right)
\cup \left( B_\alpha\setminus B\left(\xi_\alpha(t_{i-1}),\frac{\lambda}{16}\right)\right),$$
noticing that $|X(t)-Y(t)|\geq \lambda/48$ when $y$ belongs to the first set and ignoring the effect of the charges 
in the second set.
The conclusion follows.
\end{proof}

Finally, we can compare the energies as follows
\begin{lemma}
 \label{lemma:comp-energ}
There exists a constant $C_4$ such that for $t\in [0,T]$ we have for all $\alpha,\beta$
\begin{equation*}
\sqrt{h_\beta}(X,V,t)\leq \sqrt{h_\alpha}(X,V,t)+C_4,\quad \forall X\in B\left(\xi_\alpha(t),\frac{7\lambda}{48}\right),\quad\forall V\in\R^3.
\end{equation*}
 \end{lemma}

\begin{proof}  
Our aim is to show that
\begin{equation*}
 \frac{2}{|X-\xi_\beta|}+|\eta_\beta|^2-2V\cdot \eta_\beta\leq 
\frac{2}{|X-\xi_\alpha|}+|\eta_\alpha|^2-2V\cdot \eta_\alpha +2C_4^2+\sqrt{2h_\alpha}C_4,
\end{equation*}
which holds if
\begin{equation*}
  |V|\leq \frac{1}{2|\eta_\alpha-\eta_\beta|}\left(|\eta_\alpha|^2-|\eta_\beta|^2+\frac{2}{|X-\xi_\alpha|}-\frac{2}{|X-\xi_\beta|}
 +2C_4^2+\sqrt{2h_\alpha}C_4\right).
\end{equation*}
We have $|X-\xi_\beta|>|X-\xi_\alpha|$ by virtue of \eqref{est:many1}. 
Using \eqref{ineq:H2} and \eqref{est:many2} we find that the right-hand side in the above inequality is larger than
\begin{equation*}
 \frac{1}{4\sqrt{2H}}\left( -2 H+2C_4^2+2|V|C_4\right),
\end{equation*}
which is larger than $|V|$ provided that $C_4\geq 2\sqrt{2H}$.
\end{proof}

\medskip

We may finally turn to the

\textbf{Proof of Proposition \ref{prop:mainQi-m}.}

\medskip

Let $\overline{t}\in [t_{i-1},t_i]$ and $\beta\in \{1,\ldots,N\}$ such that
\begin{equation*}
 Q_i^2=h_\beta\left(X(\overline{t}),V(\overline{t}),\overline{t}\right).
\end{equation*}
There are three possibilities.

\medskip

If $x\in B_\beta$ then $X(t)\in B(\xi_\beta(t),7\lambda/48)$ on $[t_{i-1},t_i]$. In view of \eqref{est-m:der}, stability
for $\sqrt{h_\beta}$ holds on $[t_{i-1},t_i]$. In particular, we may choose $K_2$ sufficiently large so that $\sqrt{h_\beta}(x,v,t_{i-1})
\geq Q_i/2$. Hence Lemma \ref{prop:alpha} yields 
\begin{equation}
\label{est-m:8}
 Q_i\leq Q_{i-1}+C_2Q\Delta T.
\end{equation}

If $x\in \Omega$ we use Lemma \ref{prop:away} for $\sqrt{h}_\beta$ to obtain
\begin{equation}
\label{est-m:9}
 Q_i\leq Q_{i-1}+C_3Q\Delta T.
\end{equation}

Finally, if $x\in B_\alpha$ for some $\alpha\neq\beta$ then stability
for $\sqrt{h_\alpha}$ holds on $[t_{i-1},t_i]$ by \eqref{est-m:der}. Since by Lemma \ref{lemma:comp-energ} at time $\overline{t}$ we have
$$\sqrt{h_\alpha}(X(\overline{t}),V(\overline{t}),\overline{t})\geq \sqrt{h_\beta}(X(\overline{t}),V(\overline{t}),\overline{t})-C_4\geq \frac{3Q_i}{4}$$
 provided that $Q_i\geq K_3\geq 4C_4$, we may choose $K_2$ such that at time $t_{i-1}$ it holds $\sqrt{h_\alpha}(x,v,t_{i-1})\geq Q_i/2$. Hence we may apply Lemma \ref{prop:alpha} to 
$\sqrt{h_\alpha}$. Relying then once more on Lemma \ref{lemma:comp-energ} at time $\overline{t}$ we get
\begin{equation}
\label{est-m:10}
Q_i\leq  Q_{i-1}+C_4+C_2 Q \Delta T= Q_{i-1}+(C_4K_2+C_2) Q \Delta T.
\end{equation}

Setting $C_1=\max(C_2,C_3,C_4K_2+C_2)$ the conclusion follows from \eqref{est-m:8}, \eqref{est-m:9} and \eqref{est-m:10}.

\section{Final remarks and comments}

The first comment concerns the regularity of the solution we have
constructed in Theorems \ref{thm:main1} and \ref{thm:main-m}. The fact that $t \mapsto f(t)$
propagates the $C^k$ regularity of the initial condition is standard; this is a consequence of the almost-Lipschitz
regularity of the electric field $E$ (see \eqref{eq:almost-lipschitz}). In particular if
$f_0 \in C^1$ the solution we have constructed is a classical solution to Problem \eqref{eq:system}.

\medskip

The second comment regards our hypotheses. We assumed the charges with an
island of vacuum and $f_0$ compactly supported in $v$. Of course the total
energy can be finite also in absence of these hypotheses. However when
the trajectories can be arbitrarily close to the charges  or when they have
arbitrarily large velocity, there is not any uniform upper bound for the initial
energy of all the trajectories $Q_0$. Therefore our method requires new technical ideas to
take into account these large tails. 

\medskip

Finally we mention that the case of an interaction plasma-charge of
different signs eludes our techniques: new ideas and estimates are needed.

\label{sec:conclusion}

\bigskip

\noindent \textbf{Acknowlegments.} Work performed under the auspices of GNFM-INDAM and the Italian Ministry of the University
(MIUR).
The second author has been partially supported by the Fondation des Sciences Math\'ematiques 
de Paris.

\medskip

\bigskip

{(C. Marchioro)}
{Dipartimento di Matematica G. Castelnuovo, Universit\`a di Roma La Sapienza, Italy}.
\emph{E-mail address}: marchior@mat.uniroma1.it.\\

(E. Miot) 
{Dipartimento di Matematica G. Castelnuovo, Universit\`a di Roma La Sapienza, Italy}.
\emph{E-mail address}: miot@ann.jussieu.fr.\\

(M. Pulvirenti) 
{Dipartimento di Matematica G. Castelnuovo, Universit\`a di Roma La Sapienza, Italy}.
\emph{E-mail address}: pulviren@mat.uniroma1.it.\\

\end{document}